\documentclass[11pt,a4paper]{amsart}
\usepackage[T1]{fontenc}
\usepackage[utf8]{inputenc}
\usepackage[english]{babel}
\usepackage{lmodern,microtype}
\usepackage[a4paper,margin=27mm,headheight=15pt]{geometry}
\usepackage{mathtools,amssymb,amsthm,mathrsfs}
\usepackage{booktabs,tabularx,array,enumitem}
\usepackage[hidelinks]{hyperref}
\hypersetup{
  pdftitle={Reconstruction of Finite Posets of Width Three},
  pdfauthor={Patrizio Cintioli},
  pdfsubject={Reconstruction of finite posets of width three},
  pdfkeywords={finite poset, finite ordered set, width three, reconstruction,
               maximal cards, orbital balance}
}
\setlist[enumerate]{itemsep=3pt,topsep=4pt}
\setlist[itemize]{itemsep=3pt,topsep=4pt}
\newtheorem{teorema}{Theorem}[section]
\newtheorem{lemma}[teorema]{Lemma}

\newtheorem{corollario}[teorema]{Corollary}
\theoremstyle{definition}
\newtheorem{definizione}[teorema]{Definition}
\newtheorem{esterno}{Reconstruction result}

\theoremstyle{remark}
\newtheorem{osservazione}[teorema]{Remark}
\DeclareMathOperator{\Aut}{Aut}
\DeclareMathOperator{\Max}{Max}

\DeclareMathOperator{\Idl}{Idl}
\DeclareMathOperator{\rk}{rk}

\newcommand{\N}{\mathbb N}
\newcommand{\Z}{\mathbb Z}
\newcommand{\Deck}{\mathcal D}
\newcommand{\MD}{\mathcal M}
\newcommand{\MMD}{\mathcal M^{*}}
\newcommand{\mset}[1]{\{\!\{#1\}\!\}}
\newcommand{\Orb}{\mathscr O}
\newcommand{\BO}{\mathrm{BO}}
\newcommand{\pt}[2]{\tau(#1,#2)}

\numberwithin{equation}{section}

\title[Reconstruction of Finite Posets of Width Three]
{Reconstruction of Finite Posets of Width Three}
\author{Patrizio Cintioli}
\address{Mathematics Division, School of Science and Technology,
University of Camerino, 62032 Camerino, Italy}
\email{patrizio.cintioli@unicam.it}



\begin{document}

\begin{abstract}
We prove that every finite ordered set of width at most three and
cardinality at least four is reconstructible from its deck.
This settles the width-three part of an open problem posed
by Schröder in \emph{Ordered Sets}; the separate question
concerning ordered sets of height one is not addressed here.
The new ingredient is a criterion for recovering the marked maximal
deck of an ordered set with exactly three maximal elements from its
unmarked maximal deck and the isomorphism type of the suborder
obtained by deleting all maximal elements. We express the criterion
as an orbital balance condition on the attachment ideals in this
common base. Necessity follows by double counting; sufficiency uses
the structure of automorphism actions in width at most three.
The full-deck theorem then follows from reconstruction results of
Schröder concerning maximal cards and marked maximal cards.
A local reattachment argument correcting one intermediate step in
the two-maximal-element case is included in an appendix.
\end{abstract}

\maketitle

\section{Introduction and main theorem}
\label{sec:intro}

A reconstruction problem asks how much of a finite structure survives
when one deletes a single point in every possible way. Instead of being
given the structure itself, one is given the collection of all of its
one-point deletions and asks whether the original object can be recovered
up to isomorphism.  This question originated in graph theory, where it is
traditionally traced back to Kelly and Ulam in the 1940s; the familiar language of
``cards'' was subsequently popularized by Harary. The analogous question for ordered
sets asks whether the global order can be recovered, up to isomorphism,
when one is given all one-point-deleted suborders without labels. For
historical background on reconstruction and its passage from graphs to
ordered sets, see \cite[Remark~4 following Section~1.5]{S16} and
\cite{R05}.

From a combinatorial point of view, the reconstruction problem is
a local-to-global question: each
card has forgotten one point and all information involving that point,
while the deck retains the overlapping information supplied by all
deletions. Two broad approaches have guided work on the problem. One is
to prove reconstructibility for successively larger classes of ordered
sets; the other is to recover sufficiently many invariants from the
deck until the original order is forced. Both viewpoints will appear
below. In the width-three case, the main difficulty is not the recovery
of elementary numerical parameters, but the consistent identification
of how the maximal elements are attached to the part of the order that
survives all maximal deletions.

For a finite ordered set $P$, a \emph{card} is a suborder $P\setminus\{x\}$.
Its deck is the multiset
\[
 \Deck(P)=\mset{[P\setminus\{x\}]:x\in P},
\]
where $[X]$ denotes the isomorphism type of $X$. The cards are unlabeled;
multiplicities are retained. The ordered set $P$ is \emph{reconstructible}
if every finite ordered set $Q$ with $\Deck(Q)=\Deck(P)$ is isomorphic
to $P$. For a survey of reconstruction problems and results for ordered
sets, see \cite{R05}.

The width-three reconstruction problem was already investigated
by Schröder in \cite{S03}. He later explicitly listed the
reconstructibility of ordered sets of width three, together with
the separate height-one case, among the open problems in
\cite[Chapter~2, Remarks and Open Problems, item~9, p.~50]{S16}.
Theorem~\ref{thm:target} settles the width-three part of this
problem. The question for ordered sets of height one and arbitrary
width is not addressed here. Our proof follows the approach
suggested in \cite{S16}, using the reconstruction tools of
\cite{S03} and the subsequent results of \cite{S25}.

Using the reconstruction results specified in
Section~\ref{sec:external}, we obtain the following result.

\begin{teorema}[Main theorem]\label{thm:target}
Let $P$ be a finite ordered set with $|P|\ge4$ and $w(P)\le3$. For every finite ordered set $Q$, the equality $\Deck(P)=\Deck(Q)$ implies $P\cong Q$.
\end{teorema}

The new step concerns an ordered set $P$ of width at most three with
exactly three maximal elements. Write
\[
 B_P=P\setminus\Max(P).
\]
The \emph{unmarked maximal deck} records the types of the three cards
$P-m$, for $m\in\Max(P)$, with multiplicities. The \emph{marked maximal
deck} also distinguishes $B_P$ within each card. We show that the
unmarked maximal deck and $[B_P]$ determine this marked deck. This is a
statement about the data associated with an actual ordered set; it
does not assume that those data have already been recovered from the
full deck.

To recover the marked deck, we compare candidate markings through the
orbits of their attachment ideals under $\Aut(B_P)$. A double-counting
identity gives the necessary \emph{orbital balance} condition. Its
sufficiency is Theorem~\ref{thm:BO}; the recovery statement follows in
Corollary~\ref{cor:recovery}. The group-theoretic step uses small
diagonal factors to pass from pairwise to simultaneous compatibility.
The width bound is essential to this step: the compatibility principle
already fails on a four-point antichain
(Remark~\ref{rem:widthessential}).

Theorem~6.6 of \cite{S03} shows that an ordered set of width three is
reconstructible from its full deck together with its marked maximal
deck. Corollary~8.1 of \cite{S25} shows that, in width three, the
unmarked maximal and minimal decks are reconstructible. Finally,
Theorem~7.4 of \cite{S03} settles the case of two maximal elements.
Thus the remaining three-maximal-element problem is to pass from the
unmarked maximal deck, together with the common base, to the marked
maximal deck. This is the role of orbital balance.

The two parts of the argument meet in
Section~\ref{sec:reconstruction}: the reconstruction results provide
$[B_P]$ and the unmarked maximal deck, orbital balance recovers the
marked maximal deck, and the full deck together with that marked deck
determines $P$. The full deck is retained in the last step; we do not
claim that the marked maximal deck alone determines every ordered set
of width three.

Sections~\ref{sec:balance}--\ref{sec:sufficiency} develop the
orbital-balance argument independently of the reconstruction results.
Appendix~\ref{app:repair} records a local rank-based reattachment
argument related to the two-maximal-element case.

\section{Conventions and background results}
\label{sec:external}

We first fix the conventions used throughout the paper and then record
the previously known reconstruction results needed for the final passage
from the orbital-balance theorem to reconstruction from the full deck.
The internal argument in Sections~\ref{sec:balance}--\ref{sec:sufficiency}
is independent of these results.

\subsection{Conventions}
All ordered sets are finite unless explicitly stated otherwise, and all suborders are induced. An isomorphism preserves and reflects the order. We write $X-x$ for $X\setminus\{x\}$ and $X-S$ for $X\setminus S$. When $X$ and a point $x\notin X$ lie in a specified ambient ordered set,
$X+x$ denotes the induced suborder on $X\cup\{x\}$; the notation is iterated when several points are added. It does not denote an ordinal sum.
An empty base is allowed in the structural lemmas.

The width $w(X)$ is the maximum cardinality of an antichain;
we put $w(\varnothing)=0$. An ordered set is connected when
its comparability graph (the undirected graph joining distinct
comparable elements) is connected. Equivalently, its Hasse
diagram is connected when edge orientations are ignored.

The rank $\rk_X(x)$ is the maximum length of a chain ending
at $x$, where length is the number of elements minus one. Put
\[
 \downarrow_Xx=\{y\in X:y\le x\},\qquad
 \ell_X(x)=|\downarrow_Xx|.
\]
Thus $\ell_X(x)$ is the cardinality of the principal ideal
of $x$, including $x$ itself; it need not equal
$\rk_X(x)+1$.

A \emph{down-set} $D\subseteq X$ satisfies
$x\in D,\ y\le x\Rightarrow y\in D$;
an \emph{up-set} is defined dually.
A subset is said to be preserved by an automorphism if it is
mapped onto itself setwise; its elements need not be fixed
individually. The attachment ideals used below are
\emph{strict predecessor sets}, excluding the attached
maximal element itself.

A \emph{pointed ordered set} is a pair $(X,x)$ with $x\in X$; we write $\pt{X}{x}$ for its isomorphism type, requiring isomorphisms to preserve the distinguished point. Equalities of pointed types are not meant as equalities of labeled ordered sets.

Multisets may be represented by formal sums with coefficients in $\N=\{0,1,\ldots\}$. The order between such sums is coefficientwise. When common contributions are canceled, we work equivalently in the free abelian group with coefficients in $\Z$.

\subsection{Reconstruction results used}

In the following reconstruction results, finiteness and cardinality at
least four are always understood. We state only the forms used here.
Numbered references to \cite{S25} follow arXiv:2312.15135v3.

\begin{esterno}[Reconstructible parameters]\label{ext:parameters}
Width, the number of maximal elements, and connectedness are determined by the full deck.
\end{esterno}
\noindent References: \cite[Propositions~2.23 and~2.51, Corollary~3.9]{S16}.

\begin{esterno}[Classes already reconstructible]\label{ext:easy}
Disconnected ordered sets, ordered sets with a greatest or a least element, and ordered sets of width at most two are reconstructible.
\end{esterno}
\noindent References: \cite[Propositions~1.37 and~2.51, Exercise~3-8]{S16}, with duality for a greatest element; for width two, the original result is \cite{KR96}.

\begin{esterno}[Type of the base]\label{ext:base}
For a connected ordered set $P$, the type of $B_P=P\setminus\Max(P)$ is determined by the full deck.
\end{esterno}

\noindent References: \cite[Scholium~3.11]{S16}, in its form with a
reconstruction alternative. It identifies a maximal card together with
the marked subset $B_P$, or already concludes that $P$ is
reconstructible; either alternative determines the base type. The dual
of \cite[Lemma~4.1]{S25} gives another formulation of the same
base-type conclusion.

\begin{esterno}[Unmarked maximal deck]\label{ext:unmarked}
For an ordered set of width three, the multiset
\[
 \MD(P)=\mset{[P-m]:m\in\Max(P)}
\]
is determined by the full deck.
\end{esterno}

\noindent Reference: \cite[Corollary~8.1]{S25}. The corollary
reconstructs the unmarked minimal and maximal decks separately; here we
use its maximal-deck conclusion, with multiplicities.

\begin{esterno}[Two maximal elements]\label{ext:two}
An ordered set of width three with two maximal elements is reconstructible from the full deck.
\end{esterno}

\noindent Reference: \cite[Theorem~7.4, p.~327]{S03}. Appendix
\ref{app:repair} discusses one intermediate assertion in the proof of
Proposition~7.3, Subcase~1.1, and replaces it by a rank-by-rank
reattachment argument. The statement of Theorem~7.4 is not altered.

For the last reconstruction result, distinguish the subset $B_P$ on each maximal card $P-m$. The \emph{marked maximal deck} is
\begin{equation}\label{eq:markeddeck}
 \MMD(P)=\mset{[(P-m,B_P)]:m\in\Max(P)}.
\end{equation}
An isomorphism of pairs maps the distinguished subset onto the distinguished subset, without necessarily fixing its individual points. This is the convention of the color class \emph{no} in \cite[Definition~3.3, p.~304]{S03}; distinguishing instead the surviving original maximal elements is equivalent.

\begin{esterno}[Full deck and marked deck]\label{ext:marked}
If $P$ has width three and $Q$ is a finite ordered set such that
\[
 \Deck(P)=\Deck(Q),\qquad \MMD(P)=\MMD(Q),
\]
then $P\cong Q$.
\end{esterno}

\noindent Reference: \cite[Theorem~6.6, p.~320]{S03}. The full deck
remains part of the hypothesis; the marked maximal deck alone is not
being asserted to determine the ordered set.

\begin{osservazione}[Use of the reconstruction results]
The proof of Corollary~8.1 of \cite{S25} uses the two-maximal-element
case \cite[Theorem~7.4]{S03}. We do not assume any theorem that already
reconstructs all ordered sets of width three. Results
\ref{ext:parameters}--\ref{ext:marked} are used only in
Section~\ref{sec:reconstruction}; the structural arguments in
Sections~\ref{sec:balance}--\ref{sec:sufficiency} are independent of
them.

Appendix~\ref{app:repair} concerns only one intermediate step in the
proof of Proposition~7.3 of \cite{S03}. It is not used in the proof of
the orbital-balance theorem.
\end{osservazione}

\section{Candidate colorings and orbital balance}
\label{sec:balance}

We now turn to the case of an ordered set with exactly three maximal
elements. The aim is to encode a candidate coloring of each maximal
card through the automorphism orbits of its attachment ideals in the
common base, and then to formulate a compatibility condition across the
three cards. We begin with the elementary deletion facts needed for
this encoding.

\subsection{Elementary facts about deletions}
\begin{lemma}[Ideals and upper deletions]\label{lem:upset}
If $V$ is an up-set of $X$ and $x\in X\setminus V$, then
\[
 \downarrow_{X\setminus V}x=\downarrow_Xx.
\]
Moreover, $x<y$ implies $\ell_X(x)<\ell_X(y)$; hence every class
$\{x\in X:\ell_X(x)=s\}$ is an antichain.
\end{lemma}
\begin{proof}
If $z\le x$ belonged to $V$, upward closure would imply $x\in V$. This proves the first assertion. For the second, $x<y$ implies $\downarrow_Xx\subsetneq\downarrow_Xy$.
\end{proof}

Assume now that $w(P)\le3$ and $\Max(P)=\{m_1,m_2,m_3\}$. Write
\begin{equation}\label{eq:setup}
 B=P\setminus\Max(P),\qquad C_i=P-m_i,\qquad
 D_i=\{b\in B:b<m_i\},\qquad G=\Aut(B).
\end{equation}
The $D_i$ are down-sets, possibly equal, and $D_1\cup D_2\cup D_3=B$.

\begin{lemma}[New maximal element]\label{lem:newmax}
The card $C_i$ has two or three maximal elements. If it has three, then besides the two surviving original maximal elements there is a unique point $r_i\in B$, and
\begin{equation}\label{eq:uniqueupper}
 \{x\in P:r_i<x\}=\{m_i\}.
\end{equation}
\end{lemma}
\begin{proof}
The other two maximal elements of $P$ remain maximal in $C_i$. The maximal elements of $C_i$ form an antichain and $w(C_i)\le3$. Any additional maximal element $r_i$ was not maximal in $P$, but all its strict upper bounds were deleted; the only deleted point is $m_i$.
\end{proof}

\begin{definizione}[Candidate coloring]\label{def:candidate}
A candidate coloring of the card $C_i$ is the choice of a set
$O_i\subseteq\Max(C_i)$ of cardinality two, called the set of \emph{designated old maximal elements}, such that
\[
 B_i=C_i\setminus O_i\cong B.
\]
The colored card is $(C_i,B_i)$. The authentic coloring is given by
$O_i=\Max(P)\setminus\{m_i\}$ and $B_i=B$. If $C_i$ has three maximal elements, a candidate choice is equivalent to designating one of them as the new maximal element; if it has two, the choice is forced.
\end{definizione}
We do not assume that a candidate coloring is already a reconstructible marking in the sense of Definition~3.2 of \cite{S03}. It is merely a finite pair subjected to the tests below.

\begin{lemma}[The base type forces the ideal type]\label{lem:basefilter}
Let $C$ be an ordered set with exactly three maximal elements $x,y,z$. If
\[
 C\setminus\{y,z\}\cong C\setminus\{x,z\},
\]
then $\downarrow_Cx\cong\downarrow_Cy$.
\end{lemma}
\begin{proof}
Put $K=C\setminus\{x,y,z\}$. For an ordered set $H$ define the formal sum
$I(H)=\sum_{u\in H}[\downarrow_Hu]$. The ideals of the points of $K$ do not change in the two extensions $K+x$ and $K+y$. Hence
\[
 I(K+x)=I(K)+[\downarrow_Cx],\qquad
 I(K+y)=I(K)+[\downarrow_Cy].
\]
The isomorphism of the extensions makes the left-hand sides equal. Canceling $I(K)$ in the free abelian group on isomorphism types gives
the conclusion. The isomorphism of the extensions is not required to preserve $K$.
\end{proof}

Put $d_i=\ell_P(m_i)$. If a nonauthentic coloring of $C_i$ designates the old maximal element $m_j$ as new, exchanging its role with $r_i$, then Lemma~\ref{lem:basefilter} gives
\begin{equation}\label{eq:sizeineq}
 \ell_P(r_i)=d_j<d_i.
\end{equation}
We have used Lemma~\ref{lem:upset} to pass from $C_i$ to $P$, and $r_i<m_i$ for the inequality. In particular, if $d_i$ is minimal, the coloring of $C_i$ is forced.

\begin{corollario}[Bound on the choices]\label{cor:four}
After fixing the type of $B$ and the unmarked maximal deck of $P$, there are at most four candidate assignments of colorings to the three card occurrences, before identifying assignments with the same multiset of colored types.
\end{corollario}
\begin{proof}
Order $d_1\le d_2\le d_3$. The authentic coloring always gives one choice; by \eqref{eq:sizeineq}, an additional choice is possible only for an old maximal element whose ideal is smaller than the ideal of the deleted maximal element. If the sizes are distinct, the new maximal element cannot simultaneously have the size of both old maximal elements. Thus the bounds are
\[
\begin{array}{c|c|c}
\text{sizes}&\text{choices per card}&\text{product}\\\hline
 d_1=d_2=d_3&1,1,1&1\\
 d_1=d_2<d_3&1,1,3&3\\
 d_1<d_2=d_3&1,2,2&4\\
 d_1<d_2<d_3&1,2,2&4
\end{array}
\]
Identifying isomorphic colorings or occurrences can only decrease the number.
\end{proof}

\subsection{Atoms and margins}
Let $\Idl(B)$ denote the set of all down-sets of $B$, including the empty set, and put $\Orb(B)=\Idl(B)/G$, where $G=\Aut(B)$.
Denote by $[D]_G$ the atom corresponding to the orbit of $D$, and by $\N^{(\Orb(B))}$ the free commutative monoid on these atoms. Each atom has a well-defined attachment size, equal to $|D|$.

For a candidate coloring $(C,B')$, choose an isomorphism $\lambda:B'\to B$. If $x,y$ are the two designated old maximal elements, put
\begin{equation}\label{eq:margin}
 v(C,B')=
 \bigl[\lambda(\{b\in B':b<x\})\bigr]_G+
 \bigl[\lambda(\{b\in B':b<y\})\bigr]_G.
\end{equation}
This is the \emph{orbital margin} of the card. Two choices of $\lambda$ differ by an automorphism of $B$, so the margin does not depend on the choice. The type of the colored card is the orbit of the \emph{unordered pair with multiplicity} of the two attachments. The margin retains only the two individual orbits.
Indeed, an isomorphism of colored cards restricts to an isomorphism of
the distinguished bases and therefore, after transport to $B$, sends
the unordered pair of attachment ideals, with multiplicity, to the
corresponding pair. Conversely, an automorphism of $B$ carrying one
such unordered pair to the other extends to an isomorphism of the
colored cards by matching the two designated maximal elements.

The strict ideals of $x,y$ are contained in $B'$, because $x,y$ are distinct maximal elements of $C$. Therefore the sizes of the two atoms in the margin are $\ell_C(x)-1$ and $\ell_C(y)-1$, independently of the chosen isomorphism.

\begin{lemma}[Pointed interpretation of atoms]\label{lem:pointed}
The atom $[D]_G$ is equivalent to the pointed type obtained by adjoining to $B$ a new maximal element $d$ above exactly $D$ and distinguishing $d$. In particular, if $x,y$ are the designated old maximal elements of $(C,B')$, then the margin is represented by the multiset
\[
 \mset{\pt{C-y}{x},\pt{C-x}{y}}.
\]
\end{lemma}
\begin{proof}
After deleting the distinguished point, a pointed isomorphism restricts to an isomorphism of the bases that maps one attachment to the other. Conversely, such an isomorphism of the bases extends by sending the added point to the added point. This gives the bijective correspondence with the orbits under $G$.
\end{proof}

In what follows we identify atoms with their pointed representations from Lemma~\ref{lem:pointed}. For the authentic cards, writing
\[
 \alpha=[D_1]_G,\qquad \beta=[D_2]_G,\qquad \gamma=[D_3]_G,
\]
the margins are
\begin{equation}\label{eq:authmargins}
 v_1=\beta+\gamma,\qquad v_2=\alpha+\gamma,\qquad
 v_3=\alpha+\beta.
\end{equation}

\begin{definizione}[Orbital balance]\label{def:BO}
Three candidate margins $v'_1,v'_2,v'_3$ satisfy \emph{orbital balance}, abbreviated $\BO$, if
\begin{equation}\label{eq:BO}
 \begin{gathered}
 R=v'_1+v'_2+v'_3\text{ has all coefficients even},\\
 h=\tfrac12R,\qquad v'_i\le h\quad(i=1,2,3).
 \end{gathered}
\end{equation}
\end{definizione}

\begin{lemma}[Necessity of the balance condition]\label{lem:necessary}
The authentic colorings satisfy $\BO$.
\end{lemma}
\begin{proof}
Each original maximal element survives on exactly two cards, and its attachment ideal does not change. By \eqref{eq:authmargins}, $R=2(\alpha+\beta+\gamma)$, and every $v_i$ is at most $\alpha+\beta+\gamma$.
\end{proof}

\begin{osservazione}[Not an arbitrary gluing criterion]
The condition $\BO$ alone does not guarantee that arbitrary colored data arise from an ordered set. For example, let the base be a three-point antichain, and let $J,K$ be distinct two-point down-sets. The attachment pairs
\[
 \mset{J,J},\qquad \mset{J,J},\qquad \mset{J,K}
\]
have margins $2\omega,2\omega,2\omega$, where $\omega$ is the unique orbit of two-point subsets. They satisfy $\BO$, but they are not the three pairs coming from a triple: two equal diagonal attachment pairs would force the third pair to be equal as well. The sufficiency result proved below requires the unmarked cards of one actual ordered set $P$ and candidate bases isomorphic to $B_P$.
\end{osservazione}

\section{Automorphism actions in width at most three}
\label{sec:groups}

The sufficiency of orbital balance requires turning several pairwise
compatibility conditions on attachment ideals into a single
automorphism of the base. In width at most three this is possible
because the automorphism action has a particularly restricted form.
This section isolates the group-theoretic mechanism behind that
simultaneous compatibility.

\begin{lemma}[Small diagonal factors]\label{lem:diagonal}
Let $X$ be a finite ordered set of width at most three. Up to renumbering within the orbits, the action of $\Aut(X)$ is a direct product, on disjoint supports, of diagonal actions of
\[
 S_2,\qquad S_3,\qquad C_3,
\]
together with fixed points. A diagonal factor acts by the same permutation on several copies of a two- or three-point set, with respect to fixed identifications.
\end{lemma}
\begin{proof}
Every orbit is an antichain. Indeed, if two points in the same orbit satisfied $x<g(x)$, iteration of the finite permutation $g$ would produce a strict cyclic chain. Hence orbits have cardinality one, two, or three.

Consider two distinct orbits $O,O'$. Rank is constant on each orbit, so if there are comparabilities between them, they have only one orientation. The relation between the two sets is biregular: transitivity of the action on each orbit makes the number of neighbors constant on each side. If the cardinalities are two and three, the number of edges satisfies
\[
 2d=3e,\qquad 0\le d\le3,\quad 0\le e\le2.
\]
Only zero or six edges are possible. The relation is empty or complete. The same is true if one orbit is a singleton; such relations do not couple the permutations of the orbits.

If $|O|=|O'|=k\in\{2,3\}$, every nonempty, noncomplete relation is a perfect matching or, when $k=3$, the complement of a perfect matching. It determines a bijection $f:O\to O'$. Preserving the relation is equivalent to imposing
\begin{equation}\label{eq:matching}
 \sigma_{O'}=f\sigma_Of^{-1}.
\end{equation}

Form the graph of the nonsingleton orbits, joining two orbits when a nonuniform relation imposes such a bijection. Every connected component consists of orbits of the same cardinality $k$. Choose a root orbit and a spanning tree. By transporting coordinates along the tree, all permutations in the component are determined by one $\sigma\in S_k$. Every remaining edge requires $\sigma$ to commute with the coordinate permutation obtained around the corresponding cycle. Thus the allowed choices are exactly
\[
 C_{S_k}(H),
\]
where $H$ is the subgroup generated by these cycle permutations.

To justify exactness, necessity of \eqref{eq:matching} is not enough. Conversely, every system of orbit permutations satisfying all these constraints preserves every order relation between every pair of orbits. Within an orbit there is only equality; between distinct components the relations are empty or complete. Such a system is therefore an automorphism of the whole of $X$. It also follows that the choices on distinct components are independent.

The initial orbits are the actual orbits of $\Aut(X)$; hence the centralizer on every root must be transitive. For $k=2$ it is $S_2$. For $k=3$, the centralizers of the subgroups $1,C_2,C_3,S_3$ are respectively $S_3,C_2,C_3,1$; only $S_3$ and $C_3$ are transitive on the three points. This proves the asserted factors and their diagonal action.
\end{proof}

\begin{lemma}[Simultaneous compatibility of subsets]\label{lem:pairwise}
Let $X$ have width at most three, and let $A_1,\ldots,A_t,A'_1,\ldots,A'_t$ be subsets of $X$. Suppose that for every $i,j\in\{1,\ldots,t\}$, including $i=j$, there exists $g_{ij}\in\Aut(X)$ such that
\[
 g_{ij}(A_i)=A'_i,\qquad g_{ij}(A_j)=A'_j.
\]
Then there exists $g\in\Aut(X)$ such that $g(A_i)=A'_i$ for every $i$.
\end{lemma}
\begin{proof}
Use the decomposition of the action in Lemma~\ref{lem:diagonal}. It is enough to solve all conditions on each factor and then combine the solutions on disjoint supports.

In a diagonal $S_3$ factor, on each three-point copy a nonempty proper part of a subset is a singleton or the complement of a singleton. Requiring this part to be mapped to the prescribed part of $A'_i$ is therefore a prescription for the image of an index in $\{1,2,3\}$. The same permutation of $\{1,2,3\}$ must be used on every copy of the factor; the indices designated by the individual parts need not be the same.

The prescriptions arising from one $A_i$ are consistent by the case $i=j$. If prescriptions arising from different subsets assigned different images to the same source index, or the same image to distinct source indices, the relevant $g_{ij}$ could not exist. Hence the total set of prescriptions is a partial bijection of the index set, extendable to a permutation. Empty or full parts are respected by individual compatibility. The $S_2$ factor is handled in the same way.

In a diagonal $C_3$ factor, if all subsets have only empty or full parts, there is no constraint. Otherwise choose an $A_i$ with a nontrivial part. Its stabilizer in $C_3$ is trivial, so the condition on the image of $A_i$ determines a unique element of the factor. Pairwise compatibility with every other $A_j$ forces the same element to satisfy the condition on $A_j$ as well.

Fixed points introduce no choices. The product of the solutions on the factors is the required automorphism.
\end{proof}

\begin{osservazione}[Tuples of subsets]\label{rem:tuples}
Lemma~\ref{lem:pairwise} also applies when each object is an ordered tuple of subsets and the image of the entire tuple is prescribed. Apply the lemma to all coordinates: pairs of coordinates within one tuple are handled by its individual compatibility, while coordinates from two tuples are handled by compatibility of the pair of tuples.
\end{osservazione}

For tuples of subsets of $X$, write $\sim_X$ for belonging to the same orbit under the coordinatewise action of $\Aut(X)$. The ordering of the coordinates is part of the data.

\begin{lemma}[Completion of a swap]\label{lem:swap}
Let $A,R,S,N\subseteq X$, where $w(X)\le3$. Suppose that automorphisms of $X$ realize the following three conditions, not necessarily by the same automorphism:
\begin{enumerate}[label=\textup{(\roman*)}]
\item the swap $R\leftrightarrow S$;
\item the cycle $R\mapsto A\mapsto S\mapsto R$;
\item $R\mapsto A$ while preserving $N$.
\end{enumerate}
Then an automorphism of $X$ swaps $A,R$ and preserves $S,N$.
\end{lemma}
\begin{proof}
If two of $A,R,S$ coincide, the cycle makes all three equal and the identity suffices. Assume they are distinct. The cycle makes the ordered pairs along each of the two orientations of the triangle with vertices $A,R,S$ equivalent. The swap $R,S$ identifies the two orientations; hence all six ordered pairs with distinct indices belong to the same orbit. The three single orbits also coincide.

Consequently, the desired swap $A\leftrightarrow R$, with $S$ fixed, respects all single and pair orbits of $A,R,S$. For pairs involving $N$, the third hypothesis gives $(R,N)\sim_X(A,N)$ and, using the inverse, the opposite correspondence. The pair $(S,N)$ is unchanged. Lemma~\ref{lem:pairwise} realizes the four prescriptions simultaneously.
\end{proof}

\begin{osservazione}[The width bound is essential]
\label{rem:widthessential}
On a four-point antichain, the ordered families
\[
 (\{1,2\},\{1,3\},\{1,4\})
 \quad\text{and}\quad
 (\{1,2\},\{1,3\},\{2,3\})
\]
are compatible on every pair under $S_4$: any two coordinates have cardinality two and intersection of cardinality one. They are not simultaneously equivalent, because the triple intersection has cardinality one in the first family and zero in the second. Thus Lemma~\ref{lem:pairwise} is not being used as a general fact about permutation groups.
\end{osservazione}

\section{Local separation of colorings}
\label{sec:local}

An orbital margin records the individual orbits of the two attachment
ideals but, in general, forgets their relative position.
For two candidate colorings of the same card, however, the width bound
recovers precisely the information that is lost. The next
lemma is the local separation statement that will allow us to pass from
equality of margins to equality of colored-card types.

\begin{lemma}[The margins distinguish colored types on the same card]\label{lem:local}
Let $C$ be an ordered set of width at most three with exactly three maximal elements $a,b,c$.
Consider the two colorings that designate $a$ and $b$, respectively, as the new maximal element.
If their bases are isomorphic and, after identifying both bases with a
common copy $B$, their orbital margins coincide in $\N^{(\Orb(B))}$, then the two colored cards
are isomorphic. Equivalently, there is an automorphism of $C$ sending $a$ to $b$.
\end{lemma}
\begin{proof}
Lemma~\ref{lem:basefilter} gives $\ell_C(a)=\ell_C(b)=s$. By Lemma~\ref{lem:pointed}, equality of the margins is equivalent to
\begin{equation}\label{eq:pointedlocal}
\begin{split}
 &\mset{\pt{C-c}{b},\pt{C-b}{c}}\\
 &\hspace{12mm}=\mset{\pt{C-c}{a},\pt{C-a}{c}}.
\end{split}
\end{equation}
There are two cases.

\medskip\noindent\textbf{Case 1: $\ell_C(c)=s$.}
Every nonmaximal point has an ideal of cardinality smaller than $s$, because it lies below a maximal point. Put $Y=C\setminus\{a,b,c\}$, and denote the strict ideals of $a,b,c$ in $Y$ by $A,B,E$, respectively. In each pointed ordered set in \eqref{eq:pointedlocal}, $Y$ is recognizable as the set of points whose ideal has cardinality smaller than $s$. Every pointed isomorphism in that formula therefore restricts to an automorphism of the same $Y$.

Equality \eqref{eq:pointedlocal} becomes an equality between the multisets of orbits of ordered pairs
\begin{equation}\label{eq:pairslocal}
 \mset{[(A,B)],[(A,E)]}
 =\mset{[(B,A)],[(B,E)]},
\end{equation}
where the orbits are under $\Aut(Y)$. The first coordinate corresponds to the nondistinguished maximal element, the second to the distinguished point.

If $A=B$, swapping $a,b$ and fixing $Y\cup\{c\}$ is already an automorphism. Otherwise choose a matching of the two multisets in \eqref{eq:pairslocal}. Under the direct matching one obtains
\[
 (A,B)\sim_Y(B,A),\qquad (A,E)\sim_Y(B,E).
\]
These equivalences and their inverses provide all pairwise conditions for swapping $A,B$ while fixing $E$. Lemma~\ref{lem:pairwise} realizes the swap.

Under the crossed matching one instead has
\[
 (A,B)\sim_Y(B,E),\qquad (A,E)\sim_Y(B,A).
\]
These are the conditions on the pairs $(A,B)$ and $(A,E)$ for the cycle $A\mapsto B\mapsto E\mapsto A$. The third follows by reversing coordinates: from the two equivalences we obtain
\[
 (B,E)\sim_Y(A,B)\sim_Y(E,A).
\]
The single-coordinate conditions also follow by projection. Lemma~\ref{lem:pairwise} realizes the cycle. Coincidences among the ideals are automatically preserved by these conditions; if a prescription forces all three ideals to be equal, any permutation of the three maximal elements extends by the identity on $Y$.

In both matchings the automorphism of $Y$ extends to $C$ sending $a$ to $b$.

\medskip\noindent\textbf{Case 2: $\ell_C(c)\ne s$.}
The ideal cardinality of the distinguished point forces the matching in \eqref{eq:pointedlocal}. Hence there exist
\begin{equation}\label{eq:phipsi}
 \phi\in\Aut(C-c),\quad\phi(b)=a,
 \qquad
 \psi:C-b\longrightarrow C-a,\quad\psi(c)=c.
\end{equation}

Let $L=\{x\in C:\ell_C(x)=s\}$. It is an antichain containing $a,b$, so it has two or three points. In the latter case write $L=\{a,b,t\}$. The point $t$ is not maximal in $C$ and cannot lie below $a$ or $b$, which have the same ideal cardinality; hence $t<c$.

Put
\[
 Y=C\setminus(L\cup\{c\}).
\]
Deleting maximal elements preserves the ideals of the surviving points. The map $\phi$ preserves $L$, whereas $\psi$ maps $L\setminus\{b\}$ onto $L\setminus\{a\}$ and preserves $c$. Thus both maps restrict to automorphisms of $Y$. Moreover $\psi(a)=b$ and, if $t$ exists, $\psi(t)=t$: among the present points in the size-$s$ class, $t$ is the unique nonmaximal one, since it lies below $c$.

For $x\in L$ record both profiles in $Y$:
\[
 \xi_x=\bigl(\{y\in Y:y<x\},\{y\in Y:x<y\}\bigr),
 \qquad E=\{y\in Y:y<c\}.
\]
The restriction of $\psi$ maps $\xi_a$ to $\xi_b$, preserves $\xi_t$ when present, and preserves $E$.

The permutation induced by $\phi$ on $L$ maps $b$ to $a$. If it is the swap of $a,b$, with $t$ fixed if present, it already realizes all pairwise conditions on the profiles for the desired swap. If it is a 3-cycle, the ordered pairs of the three profiles have at most the two orbits corresponding to the cyclic orientations. The restriction of $\psi$ provides
\[
 (\xi_a,\xi_t)\sim_Y(\xi_b,\xi_t),
\]
which identifies the two orientations. Again all pairwise conditions for swapping $\xi_a,\xi_b$ and preserving $\xi_t$ are compatible.

For pairs involving $E$, use $\psi|_Y$ and its inverse for the profiles $\xi_a,\xi_b$; for $(\xi_t,E)$ the identity suffices. Apply Lemma~\ref{lem:pairwise} and Remark~\ref{rem:tuples}. We obtain an automorphism of $Y$ swapping $\xi_a,\xi_b$ and preserving $\xi_t,E$.

Extend it by swapping $a,b$ and fixing $c$ and, if present, $t$. The class $L$ is an antichain; between $L$ and $c$ the only possible relation is $t<c$, which is preserved. All relations with $Y$ are recorded by the two coordinates of the profiles and by $E$. The extension is therefore an automorphism of $C$ sending $a$ to $b$.
\end{proof}

\section{Transitive actions on three maximal elements}
\label{sec:transitive}

We next study the symmetric situation in which the three maximal
elements are related by deletion isomorphisms. The resulting
transitivity and extension properties will be used to control the
remaining freedom in candidate colorings in the proof of sufficiency.

\begin{lemma}[Transitivity from three deletions]\label{lem:transitive}
Let $T$ be an ordered set of width at most three with exactly three maximal elements $p,q,r$, whose principal ideals have the same cardinality. Suppose that $T-p,T-q,T-r$ are isomorphic and that $q,r$ belong to the same orbit of $\Aut(T)$. Then $\Aut(T)$ is transitive on the three maximal elements.
\end{lemma}
\begin{proof}
If $p$ already lies in the orbit of $q$, the conclusion follows. Otherwise the orbit of $q$ is $\{q,r\}$; there is an automorphism of $T$ swapping $q,r$ and fixing $p$.

Let $Y=T\setminus\{p,q,r\}$ and let $A_p,A_q,A_r$ be the attachment ideals in $Y$. As in the first case of Lemma~\ref{lem:local}, the deletion isomorphisms preserve $Y$: the two surviving original maximal elements have the common ideal cardinality, whereas every point of $Y$ has smaller ideal cardinality. Hence the three unordered pairs with multiplicity of the attachments belong to the same orbit under $\Aut(Y)$.

The swap of $q,r$ makes the ordered orbit of $(A_q,A_r)$ symmetric. Equality of the three unordered-pair orbits then implies that all six ordered pairs with distinct indices belong to the same orbit; the single orbits coincide as well. Lemma~\ref{lem:pairwise} realizes a swap of $A_p,A_q$ with $A_r$ fixed. It extends to $T$ by swapping $p,q$, contradicting the assumption that $p$ did not lie in the orbit of $q$.
\end{proof}

\begin{lemma}[Determination of the third subset in a cyclic triple]\label{lem:third}
Let $X$ have width at most three and let $A_0,A_1,A_2\subseteq X$. Suppose there exists $g\in\Aut(X)$ with
\[
 g(A_i)=A_{i+1\bmod 3}.
\]
Every automorphism of $X$ that preserves the multiset $\mset{A_1,A_2}$ also preserves $A_0$.
\end{lemma}
\begin{proof}
Use Lemma~\ref{lem:diagonal} and restrict the subsets to the support of each factor. If the projection of $g$ has order one or two, the condition $g^3(A_i)=A_i$ implies that the three restrictions coincide; preserving the pair therefore preserves the third term as well.

It remains to consider a projection of order three, in $C_3$ or $S_3$. The restrictions are either all equal, in which case there is nothing to prove, or they are three distinct patterns cyclically permuted. In $C_3$, a noninvariant pattern has trivial stabilizer, and a nonidentity element cycles all three patterns. Hence the stabilizer of a pair of distinct patterns is trivial.

Now consider a diagonal $S_3$ factor. The stabilizer of a subset pattern on the copies is $S_3$, a point stabilizer $S_2$, or the trivial group. Indeed, the nontrivial parts of the pattern are singletons or complements of singletons; preserving them means fixing all indices they designate.

In the $S_3$ case the pattern is invariant, already handled. In the $S_2$ case all nontrivial parts designate one index. The three cyclic patterns designate the three distinct indices; preserving the pair of patterns that designate two indices forces the remaining index to be fixed, and therefore preserves the third pattern.

In the trivial-stabilizer case, write the patterns as $U,gU,g^2U$ in the factor under consideration. The unique element sending $gU$ to $g^2U$ is $g$, and this element does not send $g^2U$ to $gU$. An automorphism therefore cannot swap the two patterns in the pair. If it fixes both, it is the identity. In either case it preserves $U$.

The conclusion holds on every support and on the fixed points, hence on all of $X$. We have not required $g$ to have order three on the whole of $X$.
\end{proof}

\begin{corollario}[Extension from a card]\label{cor:extension}
Let $T$ be an ordered set of width at most three with three maximal elements on which $\Aut(T)$ acts transitively. For every $p\in\Max(T)$, every automorphism of $T-p$ extends to an automorphism of $T$ fixing $p$.
\end{corollario}
\begin{proof}
The three maximal ideals have the same cardinality. Put $Y=T\setminus\Max(T)$. Every automorphism of $T-p$ preserves $Y$ and the unordered pair of attachment ideals of the two surviving original maximal elements.

The image of $\Aut(T)$ on the three maximal elements is a transitive subgroup of $S_3$, and therefore contains a 3-cycle. The restriction to $Y$ of a lift of such a cycle cyclically permutes the three attachment ideals. By Lemma~\ref{lem:third}, the automorphism of the card also preserves the attachment ideal of the deleted maximal element. It therefore extends by fixing $p$.
\end{proof}

\section{Sufficiency of orbital balance}
\label{sec:sufficiency}

We now prove that orbital balance is not only necessary but sufficient
for candidate colorings arising from the three maximal cards of a single
ordered set. The proof first excludes a single effective modification
and then analyzes the only possible configurations with two effective
modifications. The conclusion is an equality of multisets of
colored-card types; individual card occurrences need not be recovered
one by one.

Retain the notation \eqref{eq:setup} and \eqref{eq:authmargins}.

\begin{lemma}[A single effective modification is impossible]\label{lem:onemod}
If a candidate assignment differs from the authentic assignment, as a colored type, on exactly one of the three card occurrences, then it does not satisfy $\BO$.
\end{lemma}
\begin{proof}
For this proof only, let $\alpha,\beta$ be the atoms of the two surviving maximal elements on the modified card, and let $\gamma$ be the atom of the deleted maximal element. The other two margins sum to $\alpha+\beta+2\gamma$. If $v'$ is the modified margin, parity in $\BO$ requires every coefficient of $\alpha+\beta+v'$ to be even.

If $\alpha\ne\beta$, since $v'$ has total mass two, necessarily $v'=\alpha+\beta$. Lemma~\ref{lem:local} then makes the two colored types on the card equal, contrary to the hypothesis.

If $\alpha=\beta$, parity forces $v'=2\delta$ for an atom $\delta$.

If $\alpha=[D_j]_G$ and $\beta=[D_k]_G$, then
$\alpha=\beta$ implies $|D_j|=|D_k|$, hence $d_j=d_k$.
Half of the total sum is $h=\alpha+\delta+\gamma$, and the domination condition $v'\le h$ requires
\[
 2\delta\le\alpha+\delta+\gamma.
\]
Consequently $\delta=\alpha$ or $\delta=\gamma$. The modified coloring exchanges the roles of the authentic new maximal element and one of the old ones. By \eqref{eq:sizeineq}, all three maximal elements of this card have the ideal cardinality of the old maximal elements, strictly smaller than the ideal cardinality of the deleted point. Thus the attachment size of $\delta$ is that of $\alpha$, smaller than that of $\gamma$. The alternative $\delta=\gamma$ is impossible. Hence $v'=2\alpha$, and Lemma~\ref{lem:local} applies again.
\end{proof}

\begin{teorema}[Sufficiency of $\BO$]\label{thm:BO}
Let $P$ be a finite ordered set of width at most three with three
maximal elements $m_1,m_2,m_3$, and put
$C_i=P-m_i$ for $i=1,2,3$. On each concrete card $C_i$ choose a
candidate coloring, in the sense of Definition~\ref{def:candidate},
with base isomorphic to $B=P\setminus\Max(P)$.
If the three margins satisfy $\BO$, then the multiset of colored-card types is $\MMD(P)$.
\end{teorema}
\begin{proof}
By Lemma~\ref{lem:necessary}, the authentic assignment satisfies $\BO$. Suppose there is an assignment satisfying $\BO$ but producing a multiset of colored types different from the authentic one.

If a coloring already has the authentic type on its occurrence, replace it by the authentic coloring. This replacement changes neither the margin nor the multiset of colored types. After this normalization, every nonauthentic coloring is genuinely different in type.

A card deleting a maximal element with minimal $d_i$ has forced coloring. If the minimum were repeated, there would be at most one effective modification, excluded by Lemma~\ref{lem:onemod}. The same contradiction arises if there is only one effective modification for any other reason. We may therefore renumber the maximal elements so that
\begin{equation}\label{eq:abc}
 d_1=a<d_2=b\le d_3=c,
\end{equation}
with the first card authentic and the other two effectively modified.

\subsection*{Reduction of the margins}
The first card has margin $v'_1=\beta+\gamma$. By \eqref{eq:sizeineq}, the second card can exchange its new maximal element only with $m_1$. Denote this new maximal element by $r$; then
\begin{equation}\label{eq:r}
 \ell_P(r)=a,\qquad \{x\in P:r<x\}=\{m_2\}.
\end{equation}
The third card has a new maximal element $s$, distinct from $r$, whose unique strict upper bound is $m_3$. The modification may involve $m_1$, or $m_2$ when $b<c$.

Write
\begin{equation}\label{eq:modifiedmargins}
 v'_2=\delta_2+\eta_2,\qquad v'_3=\delta_3+\eta_3,
\end{equation}
where the attachment sizes are respectively $a-1,c-1$ on the second card and $a-1,b-1$ on the third. Indeed, on the third card either $m_2$ is retained and the new $s$ has ideal cardinality $a$, or $m_1$ is retained and the new $s$ has ideal cardinality $b$.

The atoms $\delta_2,\delta_3$ are the only two occurrences with attachment size $a-1$. Parity therefore forces
\begin{equation}\label{eq:deltaequal}
 \delta_2=\delta_3=: \delta.
\end{equation}
Moreover,
\begin{equation}\label{eq:etaequal}
 \mset{\eta_2,\eta_3}=\mset{\beta,\gamma}.
\end{equation}
If $b<c$, parity at the two distinct attachment sizes gives individually $\eta_2=\gamma$ and $\eta_3=\beta$. If $b=c$ and $\beta\ne\gamma$, parity gives \eqref{eq:etaequal}. Finally, if $b=c$ and $\beta=\gamma$, parity forces $\eta_2=\eta_3=\eta$; domination of the fixed margin $v'_1=2\beta$ by half the total sum gives $2\beta\le\beta+\eta$ at attachment size $b-1$, and hence $\eta=\beta$.

\subsection*{Case I: both modifications involve
\texorpdfstring{$m_1$}{m1}}
In this case $\ell_P(s)=a$ and the unique strict upper bound of $s$ is $m_3$. Put
\[
 T=P\setminus\{m_2,m_3\}=B+m_1.
\]
The three points $m_1,r,s$ are maximal in $T$ and have ideals of cardinality $a$. There are no other maximal elements, by the width bound.

The two candidate bases are $T-r$ and $T-s$, respectively; the authentic base is $T-m_1=B$. Hence
\[
 T-m_1\cong T-r\cong T-s.
\]
By the pointed interpretation, the small atoms of the modified cards are $\pt{T}{r}$ and $\pt{T}{s}$. Their equality in \eqref{eq:deltaequal} puts $r,s$ in the same orbit of $\Aut(T)$. Lemma~\ref{lem:transitive} gives transitivity on the three maximal elements. Hence
\begin{equation}\label{eq:deltaalphaI}
 \delta=\pt{T}{r}=\pt{T}{m_1}=\alpha.
\end{equation}

By Corollary~\ref{cor:extension}, every automorphism of $B=T-m_1$ extends while fixing $m_1$; hence it preserves $D_1$. The orbit $[D_1]_G=\alpha$ is therefore a singleton as an orbit of subsets of $B$.

An attachment pair containing an attachment from the orbit $\alpha$ is thus determined, up to $G$, by the single orbit of the other attachment. Indeed, the attachment from the orbit $\alpha$ must be the subset $D_1$ itself, and every automorphism realizing the image of the other attachment preserves $D_1$.

After transporting each candidate base to $B$, both modified cards have $D_1$ as one of their attachment ideals.
By \eqref{eq:etaequal}, the multiset of the other two orbit types is exactly the authentic one. It follows that the multiset of the two colored-card types is equal to the authentic one; if $b=c$, the two types may be exchanged between the two occurrences, but the deck does not change. The first card was already authentic. This contradicts the assumption of a different colored deck.

\subsection*{Case II: the third modification involves
\texorpdfstring{$m_2$}{m2}}
Now $\ell_P(s)=b<c$, so $a<b<c$. The large orbits are individually authentic:
\begin{equation}\label{eq:bigII}
 \eta_2=\gamma,\qquad\eta_3=\beta.
\end{equation}
On the second card the designated old maximal elements are $r,m_3$; on the third they are $m_1,s$. Put
\[
 U=P\setminus\{m_1,m_3\}=B+m_2.
\]
The atoms $\eta_3$ and $\beta$ are respectively $\pt{U}{s}$ and $\pt{U}{m_2}$. Their equality yields an automorphism $\phi$ of $U$ with $\phi(m_2)=s$.

\paragraph{The ideal-cardinality classes.}
The points $m_2,s$ are maximal in $U$. The orbit of $m_2$ under $\Aut(U)$ cannot contain a third point $t$: such a point would be maximal in $U$ and would have ideal cardinality $b>a$, so it would be incomparable with $m_1$ in $P$. The four points $m_1,m_2,s,t$ would form an antichain. Thus the orbit is $\{m_2,s\}$, and $\phi$ swaps them.

Let $r'=\phi(r)$. In $U$, the only strict upper bound of $r$ is $m_2$, so the only strict upper bound of $r'$ is $s$. Moreover $r'\ne r$ and $\ell_P(r')=a$, because deleting $m_1,m_3$ preserves the ideals of surviving points. The three points $m_1,r,r'$ exhaust the class of ideal cardinality $a$, which is an antichain. In $U$ this class consists only of $r,r'$, so $\phi$ swaps them. Passing from $U$ back to $P$, $r'$ cannot lie below $m_1$, which has the same ideal cardinality; its strict upper bounds in $P$ are exactly $s,m_3$.

Every additional maximal element $t$ of $U$ must lie below $m_1$: otherwise the antichain $\{m_1,m_2,s,t\}$ would again arise. Hence it has ideal cardinality smaller than $a$. Every point of $U$ lies below one of its maximal elements. A point of $U$ with ideal cardinality $b$ cannot lie strictly below $m_2$ or $s$, and cannot lie below any other maximal element, whose ideal has cardinality smaller than $a$. Therefore the size-$b$ class in $P$ is exactly $\{m_2,s\}$. Finally, $m_3$ is the unique point of $P$ with ideal cardinality $c$, since every other point lies below one of the three maximal elements and ideal cardinality strictly increases along the order.

Thus we have proved
\begin{equation}\label{eq:classesII}
 \begin{aligned}
 \{x\in P:\ell_P(x)=a\}&=\{m_1,r,r'\},\\
 \{x\in P:\ell_P(x)=b\}&=\{m_2,s\},\\
 \{x\in P:\ell_P(x)=c\}&=\{m_3\}.
 \end{aligned}
\end{equation}

\paragraph{A common lower suborder.}
Put
\begin{equation}\label{eq:K}
 K=P\setminus\{m_1,r,r',m_2,s,m_3\}.
\end{equation}
The six deleted points form an up-set. The relations among them are exactly those of the three disjoint chains
\begin{equation}\label{eq:upperchains}
 \{m_1\},\qquad r<m_2,\qquad r'<s<m_3.
\end{equation}
There are no cross-relations, by the descriptions already obtained of the strict upper bounds of $r,r',s$.

Define the six down-sets of $K$
\begin{equation}\label{eq:profiles}
\begin{aligned}
 A&=\{x\in K:x<m_1\},&R&=\{x\in K:x<r\},&S&=\{x\in K:x<r'\},\\
 M&=\{x\in K:x<m_2\},&N&=\{x\in K:x<s\},&W&=\{x\in K:x<m_3\}.
\end{aligned}
\end{equation}
We have $R\subseteq M$ and $S\subseteq N\subseteq W$. The ordered set $P$ is completely described by $K$, the chains \eqref{eq:upperchains}, and these predecessor ideals. None of the six points can lie below a point of $K$, since the six form an up-set.

All maps used below preserve $K$. Indeed, in each structure being compared, $K$ is the complement of the still-present classes of ideal cardinalities $a,b,c$. By \eqref{eq:classesII}, these classes are exactly those stated, and every deletion considered is an up-set. In particular, $\{s,m_3\}$ and $\{r,m_2\}$ are up-sets, even though they contain points that are not maximal in $P$. Lemma~\ref{lem:upset} therefore allows the original ideal cardinalities to be used in all these structures.

\paragraph{First map.}
The automorphism $\phi$ of $U$ swaps $m_2,s$ and $r,r'$. Its restriction $h\in\Aut(K)$ satisfies
\begin{equation}\label{eq:h}
 h(R)=S,\qquad h(S)=R.
\end{equation}

\paragraph{Second map.}
Equality of the small atoms in \eqref{eq:deltaequal} is equivalent to
\begin{equation}\label{eq:smallpointedII}
 \pt{P\setminus\{m_2,m_3\}}{r}
 =\pt{P\setminus\{s,m_3\}}{m_1}.
\end{equation}
Let $f$ be a pointed isomorphism between these ordered sets. By definition $f(r)=m_1$. The unique point with ideal cardinality $b$ is $s$ in the domain and $m_2$ in the codomain; hence $f(s)=m_2$. Among the three points with ideal cardinality $a$, the unique nonmaximal point is $r'$ in the domain and $r$ in the codomain. Thus $f(r')=r$ and, finally, $f(m_1)=r'$.

The restriction $k\in\Aut(K)$ realizes
\begin{equation}\label{eq:k}
 k(R)=A,\qquad k(A)=S,\qquad k(S)=R.
\end{equation}

\paragraph{Third map.}
Equality $\eta_2=\gamma$ in \eqref{eq:bigII} gives a pointed isomorphism
\begin{equation}\label{eq:largepointedII}
 \bigl(P\setminus\{m_1,m_2\},m_3\bigr)
 \cong
 \bigl(P\setminus\{r,m_2\},m_3\bigr).
\end{equation}
It fixes $s,m_3$, recognizable by the values $b,c$. In the ideal-cardinality-$a$ class, $r'$ is the unique nonmaximal point on both sides; hence the isomorphism sends $r'$ to $r'$ and $r$ to $m_1$. The restriction $g\in\Aut(K)$ satisfies
\begin{equation}\label{eq:g}
 g(R)=A,\qquad g(S)=S,\qquad g(N)=N.
\end{equation}

\paragraph{Closing the case.}
The maps \eqref{eq:h}--\eqref{eq:g} provide the swap of $R,S$, the cycle $R\mapsto A\mapsto S\mapsto R$, and the correspondence $R\mapsto A$ with $N$ preserved. By Lemma~\ref{lem:swap}, there exists $u\in\Aut(K)$ such that
\begin{equation}\label{eq:u}
 u(A)=R,\qquad u(R)=A,\qquad u(S)=S,\qquad u(N)=N.
\end{equation}

This automorphism extends to $T=P\setminus\{m_2,m_3\}$ by swapping $m_1,r$ and fixing $r',s$. Indeed, among these four points the only comparability is $r'<s$, and their predecessor sets in $K$ are exactly $A,R,S,N$. No condition on $W$ is needed, because $m_3\notin T$.

Therefore
\[
 \pt{T}{r}=\pt{T}{m_1},\qquad\text{that is,}\qquad\delta=\alpha.
\]
By \eqref{eq:bigII}, both modified margins are now individually authentic. Lemma~\ref{lem:local} then makes the colored types of the two cards authentic as well. We again contradict the initial assumption.

The two cases exhaust the possible modifications allowed by \eqref{eq:sizeineq}; the proof is complete.
\end{proof}

\begin{corollario}[Recovery of the marked deck]\label{cor:recovery}
For a finite ordered set of width at most three with three maximal elements, the type of $B_P=P\setminus\Max(P)$ and the unmarked maximal deck $\MD(P)$ determine $\MMD(P)$.
\end{corollario}
\begin{proof}
Choose a representative of the type of $B_P$ and representatives of the three occurrences in $\MD(P)$. Enumerate all candidate colorings with base isomorphic to $B_P$ and compute their margins. The authentic coloring is among the possibilities and satisfies $\BO$. By Theorem~\ref{thm:BO}, all assignments satisfying $\BO$ yield the same multiset of colored types, namely $\MMD(P)$. Corollary~\ref{cor:four} limits to four the assignments that must be subjected to this test. No complexity bound for the isomorphism checks is inferred from this.
\end{proof}

\section{Passage to the full deck}
\label{sec:reconstruction}

It remains to combine the internal recovery theorem with the
reconstruction results collected in Section~\ref{sec:external}.
After the previously known cases are handled, the only remaining situation is
width three with exactly three maximal elements.

In that case the common
base and the unmarked maximal deck are available from the full deck,
Corollary~\ref{cor:recovery} recovers the marked maximal deck, and
Result~\ref{ext:marked} completes the reconstruction.

\begin{proof}[Proof of Theorem~\ref{thm:target}]
Let $P,Q$ be finite, with $|P|\ge4$, $w(P)\le3$, and $\Deck(P)=\Deck(Q)$. Their cardinalities are equal, because the cardinality is the number of occurrences in the deck. By Result~\ref{ext:parameters}, their widths, numbers of maximal elements, and connectedness also agree.

If $P$ is disconnected, has a greatest or a least element, or has width at most two, apply Result~\ref{ext:easy}. We may therefore assume that $P,Q$ are connected of width three. If they have two maximal elements, apply Result~\ref{ext:two}. With a single maximal element, finiteness implies that it is a greatest element, so that case is already covered. The only remaining case is that of three maximal elements.

By Results~\ref{ext:base} and~\ref{ext:unmarked},
\begin{equation}\label{eq:samedata}
 B_P\cong B_Q,\qquad \MD(P)=\MD(Q).
\end{equation}
Corollary~\ref{cor:recovery} implies $\MMD(P)=\MMD(Q)$. To make the treatment of labels and multiplicities explicit, one may argue directly as follows.

Match the three occurrences of the maximal cards of $P$ with those of $Q$, using equality of the multisets in \eqref{eq:samedata}. For each match choose an unmarked isomorphism. Transport to the card of $P$ the authentic subset $B_Q$ on the matched card of $Q$. In this way we obtain three candidate colorings of cards of the actual ordered set $P$, with bases isomorphic to $B_P$.

Fix an isomorphism $\lambda:B_Q\to B_P$. It identifies the orbits of down-sets in the two bases. To compute the margin on each card of $P$, choose the identification of the transported base with $B_P$ given by $\lambda$ composed with the card isomorphism to $Q$. Then the three margins are exactly the authentic margins of $Q$, transported into the same monoid $\N^{(\Orb(B_P))}$. Hence they satisfy $\BO$. Different choices of identifications do not change the margins.

Theorem~\ref{thm:BO}, applied to $P$, says that the transported colored deck is $\MMD(P)$. On the other hand, by construction it has the same multiset of types as $\MMD(Q)$. Thus the marked decks are equal.

Finally retain the original equality $\Deck(P)=\Deck(Q)$ and apply Result~\ref{ext:marked}. We obtain $P\cong Q$.
\end{proof}

\begin{osservazione}[Limits of the statement]
The theorem concerns the full deck as a multiset. It does not discard multiplicities, does not treat reconstruction from a single subdeck, and does not concern infinite ordered sets. The internal result recovering the colorings does not assume connectedness or a fixed number of minimal elements; such hypotheses were not introduced in the balance lemmas.
\end{osservazione}

\section{Conclusions}
\label{sec:conclusions}

We have proved that every finite ordered set of width at most
three and cardinality at least four is reconstructible from
its full deck. This settles the width-three part of the open
problem formulated by Schröder in
\cite[Chapter~2, Remarks and Open Problems, item~9, p.~50]{S16}.

The main new ingredient is the sufficiency of orbital balance
for candidate colorings of the three maximal cards of an
ordered set. The restricted structure of automorphism actions
in width at most three makes it possible to recover the marked
maximal deck from the unmarked maximal deck and the
isomorphism type of the common base. This recovery concerns
the multiset of marked-card types, rather than the marking of
each individual card occurrence. Combined with the existing
reconstruction results of \cite{S03,S25}, it yields
reconstructibility from the full deck.

The compatibility principle underlying the orbital-balance
argument depends essentially on the width restriction: it
already fails for subsets of a four-point antichain. Thus the
present method does not directly extend to arbitrary-width
ordered sets. In particular, the separate reconstruction
question for ordered sets of height one and arbitrary width
is not addressed here.

\appendix
\clearpage
\section{A local reattachment argument}
\label{app:repair}

This appendix concerns a reattachment step in Subcase~1.1 of the
proof of \cite[Proposition~7.3]{S03}. We give an example showing that
the stated bound on the number of lower covers does not follow from
the immediate hypotheses, and replace that bound by uniqueness at
each rank. The example concerns the intermediate sentence, not the
reconstructibility statement of \cite[Theorem~7.4]{S03}.

The argument below repairs only this indicated step. The remaining
parts of the two-maximal-element proof are taken from \cite{S03}; in
particular, the final reconstruction statement is
\cite[Theorem~7.4]{S03}.

\subsection{The intermediate sentence and an example}
We write $x\prec y$ when $x<y$ and there is no $z$ with $x<z<y$.
The immediate hypotheses at that point are
\[
 \Max(P)=\{a,b\},\qquad w(P)=3,\qquad
 |\downarrow_Pb|\ge|\downarrow_Pa|+2,
\]
with $b$ having more than one lower cover, $C=P-b$ having exactly two maximal elements $a,c$, and $\downarrow_Ca\cong\downarrow_Cc$. The text states that $b$ has exactly two lower covers: $c$ and one further nonmaximal point. As printed, this conclusion is stronger than what follows from the immediate hypotheses.

Indeed, consider
\[
 P=\{p,q,r,u,v,a,c,b\}
\]
with the following covering relations and only the further relations forced by transitivity:
\begin{equation}\label{eq:countercovers}
 p\prec u\prec a,\quad q\prec a,\quad
 r\prec v\prec c,\quad p\prec c,\quad
 u\prec b,\quad q\prec b,\quad c\prec b.
\end{equation}
The chains $p<u<a$, $r<v<c<b$, and $\{q\}$ partition $P$, whereas $\{p,q,r\}$ is an antichain; hence $w(P)=3$. We have
\[
 \Max(P)=\{a,b\},\qquad\Max(C)=\{a,c\},
\]
\[
 \downarrow_Pa=\{p,u,q,a\},\qquad
 \downarrow_Pc=\{r,v,p,c\},\qquad
 |\downarrow_Pb|=7.
\]
The two four-point ideals are isomorphic via
$p\mapsto r$, $u\mapsto v$, $q\mapsto p$, $a\mapsto c$. Nevertheless, the lower covers of $b$ are
\[
 \{u,q,c\},
\]
which has cardinality three. This is a counterexample to the indicated sentence, not to reconstructibility of $P$ or to Theorem~7.4 of the source.

The example is not a small-cardinality phenomenon. For $L>S\ge1$, take three disjoint chains
\[
 X=(x_0<\cdots<x_{L-1}),\quad
 Q=(q_0<\cdots<q_{S-1}),\quad
 R=(r_0<\cdots<r_{L-1})
\]
and three new points $a,c,b$, adding the covering relations
\[
 x_{L-1}\prec a,\quad q_{S-1}\prec a,\quad
 r_{L-1}\prec c,\quad x_{S-1}\prec c,
\]
\[
 c\prec b,\quad x_{L-1}\prec b,\quad q_{S-1}\prec b.
\]
The chains $X+a$, $R+c+b$, and $Q$ give width at most three; their three initial points are incomparable. The ideals of $a,c$ are isomorphic and have cardinality $L+S+1$, while the ideal of $b$ has cardinality $2L+S+2$. The three lower covers of $b$ are $c,x_{L-1},q_{S-1}$. The cardinalities of these examples are $2L+S+3$ and are unbounded.

\subsection{The correct property: uniqueness by rank}
\begin{lemma}\label{lem:rankrepair}
Let $C$ be an ordered set of width at most three with exactly two maximal elements $a,c$ and $\downarrow_Ca\cong\downarrow_Cc$. Put $h=\rk_C(a)=\rk_C(c)$. For every $k<h$, there is at most one point of rank $k$ that is not below $c$, and symmetrically with $a,c$ interchanged.
\end{lemma}
\begin{proof}
Let $C_k=\{x\in C:\rk_C(x)=k\}$ and put
\[
 A_k=C_k\cap\downarrow_Ca,\qquad T_k=C_k\cap\downarrow_Cc.
\]
The ideals are down-sets, so their internal ranks agree with the ranks in $C$. The isomorphism of the ideals implies $|A_k|=|T_k|$. Every point of $C$ lies below $a$ or $c$, so $A_k\cup T_k=C_k$. Since $|C_k|\le3$,
\[
 |A_k\setminus T_k|=|T_k\setminus A_k|\le1.
\]
Indeed, if both sides were at least two, the union would have at least four points.
\end{proof}

Return to $P$ under the hypotheses of Subcase~1.1. The new maximal element $c$ of $C=P-b$ is a lower cover of $b$. Every other lower cover $x$ of $b$ is not below $c$, otherwise $x<c<b$; it is not $a$, which is maximal in $P$; hence it lies below $a$ and has rank below $h$.

Once the identity of $c$ on the card has been chosen, every further lower cover is uniquely determined by its rank, by Lemma~\ref{lem:rankrepair}. The type of the ideal of $b$ gives the ranks of its lower covers, because ranks inside ideals agree with those in the ambient ordered set.
The required ideal data are available from
\cite[Theorem~2.8]{S03}, which reconstructs the isomorphism types of
the maximal ideals from the deck.

The correct procedure is therefore to reattach $b$ above \emph{all} lower covers prescribed by those ranks, in addition to $c$. There are still only the two choices for designating $c$ among the maximal elements of the card; what changes is that there may be two additional lower covers rather than only one.

\subsection{Continuation of the subcase}
For each of the two choices, perform the reattachment and compare the card obtained by deleting the other maximal element with the known card $P-a$. Candidates with incorrect numerical parameters or an incorrect ideal of the added point may also be discarded; the authentic candidate remains.

If an automorphism of $C$ swaps $a,c$, it also transports all additional lower covers, which are determined by rank and by not belonging to the ideal of the designated new maximal element. The two extensions are therefore isomorphic. If only one choice produces the required card, the authentic choice is determined.

Suppose then that both choices survive and that no automorphism of $C$ swaps $a,c$. An isomorphism between the two reconstructed cards maps the added point $b$ to the other added point: it is distinguished by the cardinality of its ideal, larger by at least two than $|\downarrow_Ca|$. Deleting that point gives
\begin{equation}\label{eq:repairphi}
 \Phi:C-a\longrightarrow C-c,\qquad\Phi(c)=a.
\end{equation}
The second equality holds because $c$ in the domain and $a$ in the codomain are the unique points with ideal cardinality $|\downarrow_Ca|$; all other points were strictly below one of the two maximal elements.

The hypotheses of Lemma~6.3 of \cite{S03} are therefore satisfied for
$C$, with the small cases already covered by Lemma~2.11 of the same
source. More explicitly, $w(C)=w(P)=3$: in an antichain of $P$
containing $b$, one may replace $b$ by $c$, since $c<b$ and $c$ is
maximal in $C$; the other points remain incomparable with $c$. If $C$
were disconnected, its two components would each have a greatest
element and would be the two isomorphic ideals; there would then be an
automorphism swapping $a,c$, contrary to the residual assumption.

The alternatives in Lemma~6.3 are handled as follows. If
$C=U\oplus V$ is a nontrivial ordinal sum, then $c\in V$ and $c<b$, so
$P=U\oplus(V+b)$ is also an ordinal sum; by
\cite[Theorem~2.4]{S03}, $P$ is reconstructible.
If $C$ has a dominating maximal
element, equality of the ideal cardinalities of $a,c$ implies that
both lie above all of $C\setminus\{a,c\}$; they can then be swapped,
contrary to the residual assumption.
In the last alternative, Lemma~6.3 yields an automorphism of
\[
C\setminus\Max(C)
=
P\setminus\bigl(\{b\}\cup\Max(P-b)\bigr)
\]
containing a $3$-cycle. Since the gap between the ideal cardinalities
is at least two and $P-b$ is identifiable, all the hypotheses of
Lemma~7.1 of \cite{S03} are satisfied, and that lemma reconstructs
$P$.

Thus the subcase is completed without assuming that $b$ has only two
lower covers.

\section*{Acknowledgments}

The main result was discovered by ChatGPT 6 Pro (OpenAI), which was
also used in developing the proof and preparing the manuscript.
The mathematical proofs were also subjected to an independent
adversarial review using Claude Fable 5.1 (Anthropic).
The author takes full responsibility for the correctness of the
mathematical statements and proofs, for the accuracy of the citations
and attributions, and for the final content of the manuscript.

\end{document}